\documentclass[12pt]{amsart}
\usepackage{amsfonts, amsmath, amsthm}
\usepackage{graphicx}
\usepackage{fourier,microtype}
\usepackage[text={7in,10in},centering]{geometry}
\usepackage{epstopdf}

\newtheorem{pro}{Proposition}[section]

\newtheorem{thm}[pro]{Theorem}
\newtheorem{lem}[pro]{Lemma}

\theoremstyle{definition}
\newtheorem{question}[pro]{Question}
\newtheorem{questions}[pro]{Questions}
\newtheorem{dfn}[pro]{Definition}
\newtheorem{nota}[pro]{Notation}
\newtheorem{dfns}[pro]{Definitions}

\newtheorem{rmkk}[pro]{Remark}

\theoremstyle{remark}
\newtheorem*{rmk}{Remark}

\newcommand{\T}{\mathcal T}
\newcommand{\s}{\Sigma}
\newcommand{\del}{\partial}
\newcommand{\bdd}{76\tm + 26}
\newcommand{\rt}{f_{r,t}} 
\newcommand{\rqq}{f_{r,q}} 
 
\newcommand{\lrb}{{\bf lr[b]}}
\newcommand{\lry}{{\bf lr[y]}}
\newcommand{\tm}{{t}}

\def\tc[#1]{t_{c}(#1)}

\title[Hyperbolic volume and {H}eegaard distance]{Hyperbolic volume and {H}eegaard distance}

\date{\today}
\address{Department of Mathematics, Nara Women's University
Kitauoya Nishimachi, Nara 630-8506, Japan}
\address{Department of mathematical Sciences, University of
Arkansas, Fayetteville, AR 72701}
\email{tsuyoshi@cc.nara-wu.ac.jp}
\email{yoav@uark.edu}
\author{Tsuyoshi Kobayashi}
\author{Yo'av Rieck}
\thanks{The first names author was supported by Grant-in-Aid for scientific
  research, JSPS grant number 19540083. The second named author was supported
  in part by the 21st century COE program ``Constitution for wide-angle
  mathematical basis focused on knots" (Osaka City University); leader: Akio
  Kawauchi. 
}

\begin{document}


\def\vol[#1]{\mbox{\rm Vol}(#1)}

\date{\today}%
\begin{abstract}
We prove  (Theorem~\ref{thm:hyperbolic})  that there exists a constant $\Lambda > 0$ so that
if $M$ is a $(\mu,d)$-generic complete hyperbolic 3-manifold of volume $\vol[M] < \infty$ and 
$\s \subset M$ is a Heegaard surface of genus $g(\s) > \Lambda \vol[M]$,
then $d(\s) \leq 2$, where $d(\s)$ denotes the distance of $\s$ as defined by Hempel.  
The term $(\mu,d)$-generic
is described precisely in Definition~\ref{dfn:generic}; see also Remark~\ref{rmk:generic}.

The key for the proof of Theorem~\ref{thm:hyperbolic} is Theorem~\ref{thm}
which is on independent interest.
There we prove that if $M$ is a compact 3-manifold that can be triangulated
using at most $t$ tetrahedra (possibly with missing or truncated vertices),
and $\s$ is a Heegaard surface for $M$ with $g(\s) \geq \bdd$, then $d(\s) \leq 2$.

\end{abstract}
\maketitle

\section{Introduction}
\label{sec:intro}

All the manifolds considered in this paper are 3-dimensional, compact, connected,
and orientable.  By {\it hyperbolic manifold} we mean a manifold whose interior admits  
a complete finite volume Riemannian metric locally isometric to hyperbolic
3-space $\mathbb{H}^{3}$.  

It is generally agreed that the volume of a hyperbolic manifold $M$, $\vol[M]$, is a good measure 
of the complexity of $M$.  As evidence of that, arguments of M. Gromov, 
T. J\o rgensen and W. Thurston show that the hyperbolic volume is linearly equivalent
to the number of tetrahedra needed to triangulate a link exterior in $M$.
The argument is based on Thurston's notes \cite{thurston}, for a detailed
presentation see~\cite{gjt}; throughout this paper 
we follow the notation and definitions given in that paper.  
For a precise statement, let $\tc[M]$ be the smallest number of tetrahedra 
needed to triangulate $M \setminus N(L)$, where the minimum is taken over 
all links $L \subset M$ (possibly, $L = \emptyset$) and all possible triangulations.  

\begin{thm}[Gromov, J\o rgensen, Thurston]
\label{thm:jt}
There  exist constants $A,\ B > 0$ so that 
for any hyperbolic manifold $M$ the following holds:
$$A \tc[M] < \vol[M] < B \tc[M].$$
\end{thm}

\begin{rmkk}
\label{rmk:jt}
In the proof of Theorem~\ref{thm:jt} given in~\cite{gjt} it is shown that
given $\mu>0$, a Margulis constant for $\mathbb{H}^{3}$, and $d>0$,
there exists $A>0$, so that $N_{d}(M_{\geq \mu})$ can be triangulated 
using at most $\frac{1}{A} \vol[M]$ tetrahedra;
here $N_{d}(M_{\geq \mu})$ denotes the closed $d$-neighborhood of
the $\mu$-thick part of $M$.
We note that $A$ depends on $\mu$ and $d$.
\end{rmkk}

Theorem~\ref{thm:jt} implies that manifolds of low volume 
admit Heegaard splittings of low genus:
let $M$ be a hyperbolic manifold, $L \subset M$ a link, and $\mathcal{T}$
a triangulation of $M \setminus N(L)$ that realizes $\tc[M]$.
Let $\Gamma$ be $\mathcal{T}^{(1)} \cup \del (M \setminus N(L))$,
where $\mathcal{T}^{(1)}$ denotes the 1-skeleton of $\mathcal{T}$.
It is easy to see that $\del N(\Gamma)$ is a Heegaard surface for $M \setminus N(L)$ and its genus is 
at most the number of tetrahedra plus one, that is, $\tc[M] + 1$.
Since the Heegaard genus does not increase after Dehn filling we get:
$$g(M) \leq \tc[M] + 1 \leq 2\tc[M] < \frac{2}{A} \vol[M].$$  
Here and throughout this paper, $g(M)$ denotes the Heegaard genus of $M$.
The converse is false:
it is easy to construct hyperbolic manifolds of arbitrarily high volume and
Heegaard genus two (for example, consider Dehn fillings of 2-bridge knots;
see \cite{schultens2bridge}).

Our goal is to show that any Heegaard surface for
a generic hyperbolic manifold $M$ is ``simple''.  This is described precisely
in Theorem~\ref{thm:hyperbolic}, and we informally explain it here.
In~\cite{hempel} J. Hempel defined a complexity of Heegaard surfaces
which we will call the {\it distance}, denoted by $d(\s)$
(the distance, which is based on Kobayashi's idea of 
{\it height of loops}~\cite{KobayashiHeights}, is defined in
Section~\ref{sec:prelims}).  
We say that a Heegaard surface $\s$ is {\it simple} if either $g(\s)$ is low (in terms of the volume) or
$d(\s) \leq 2$.   A. Casson and C. Gordon constructed a hyperbolic manifold admitting
infinitely many Heegaard surfaces, and showed that these surfaces all have
distance at least two (in their language, are strongly irreducible).  They further
showed that there is no upper bound on the genera of these surfaces;
hence this result is best possible.

We now explain what a {\it generic} hyperbolic manifold is.
Let $X$ be a compact 3-manifold (not necessarily hyperbolic) so that $\del X$ 
consists of tori, say $T_1,\dots,T_{n}$.
Let $W$ be a manifold obtained from $X$ by Dehn filling some of its 
boundary components, say $T_{1},\dots,T_{m}$, $m\leq n$.  
Note that $X \subset W$ and any
Heegaard surface for $X$ is a Heegaard surface for $W$.
Rieck and E. Sedgwick~\cite{rieck}\cite{rs1}\cite{rs2} 
prove that on each $T_{i}$ there is a finite set of slopes, denoted by $B_{i}$, 
so that if the slope filled on each $T_{i}$ intersects
every slope of $B_{i}$ more than once, then any Heegaard surface for $W$ is a 
Heegaard surface for $X$ (after isotopy if necessary).
In that case, we say that $W$ is a {\it generic} Dehn filling of $X$.
With this in mind, we define:

\begin{dfn}
\label{dfn:generic}
Let $\mu$ be a Margulis constant for $\mathbb{H}^{3}$ and fix $d>0$.  
Let $M$ be a complete hyperbolic manifold of finite volume.
Let $N_{d}(M_{\geq \mu})$ be the closed $d$-neighborhood of the $\mu$-thick
part of $M$; for a discussion see \cite{gjt}, where it was observed that
$M$ is obtained from $N_d(M_{\geq \mu})$ by Dehn filling.
We say that $M$ is $(\mu,d)$-\emph{generic} if $M$ 
is a generic Dehn filling of $N_d(M_{\geq \mu})$.
\end{dfn}

\begin{rmkk}
\label{rmk:generic}
In an effort to justify the term ``generic'' we now sketch an argument that shows that 
for any $\mu$ and $d$, there are indeed many $(\mu,d)$-generic manifold.
Fix $V>0$.  
By Remark~\ref{rmk:jt} 
there are only finitely many topological types for the manifolds $N_d(M_{\geq \mu})$,
where $M$ ranges over all hyperbolic manifolds of volume less than $V$.  Let $X$
be one of these manifolds and denote the components of $\del X$ by $T_{1},\dots,T_{n}$.
Then for each $i$ there is a finite set of slopes of $T_{i}$, say $F_{i}$,
with the following property: as above let $W$ be a 
manifold obtained from $X$ by Dehn filling some of its 
boundary components, say $T_{1},\dots,T_{m}$, $m\leq n$, so that
slope filled is not in $F_{i}$ for all $i$.  Then
$W$ is hyperbolic, the short geodesics in $W$ coincide with the cores of the attached solid
tori, and each short geodesic has a neighborhood of radius greater than $d$.  
Thus $N_d(W_{\geq \mu}) = X$.
We conclude that if $W$ is obtained by 
filling $X$ along slopes that are not in $F_{i}$ and intersect every slope
in $B_{i}$ more than once (where $B_{i}$ was defined in the paragraph preceding
Definition~\ref{dfn:generic}), then $W$ is $(\mu,d)$-generic.
This shows that if $V$ is at least the volume of the figure eight knot exterior
(so that there are infinitely many hyperbolic
manifolds of volume less than $V$), then there are infinitely many $(\mu,d)$-generic
manifolds that have volume less than $V$.
\end{rmkk}

In this paper we prove that $(\mu,d)$-generic manifolds enjoy the following
property:

\begin{thm}
\label{thm:hyperbolic}
Let $\mu>0$ be a Margulis constant for $\mathbb{H}^{3}$ and fix $d>0$.
Then there exists  $\Lambda > 0$ so that for any complete
finite volume $(\mu,d)$-generic hyperbolic manifold $M$ and
any Heegaard surface $\Sigma$ for $M$ the following holds:

If $g(\s) > \Lambda \vol[M]$, then $d(\s) \leq 2$. 
\end{thm}

\begin{rmkk}
\label{rmk:basic}
Fix a hyperbolic manifold $M$.  It is easy to see that if
$d$ is sufficiently large or $\mu$ sufficiently small, then  $M$ is diffeomorphic 
to $N_d(M_{\geq \mu})$, and in particular 
$M$ is $(\mu,d)$-generic.  
Thus, the conclusion of Theorem~\ref{thm:hyperbolic} holds for $M$.
This has two consequences:
\begin{enumerate}
\item  It is well known that the examples of Casson and Gordon mentioned above
have distance two and arbitrarily high genus. Hence
Conclusion~(2) of Theorem~\ref{thm:hyperbolic} cannot be improved.
\item If there exists
$\Lambda$ as in Theorem~\ref{thm:hyperbolic}
that is independant of $\mu$ and $d$, then the assumption
that $M$ is $(\mu,d)$-generic can be removed.  Unfortunately, for $\Lambda$ constructed in this
paper both $\lim_{d \to \infty} \Lambda = \infty$ and $\lim_{\mu \to 0} \Lambda = \infty$
hold.   
\end{enumerate}
\end{rmkk}

Our proof of Theorem~\ref{thm:hyperbolic} uses Dehn filling and hence forces
us to assume that $M$ is $(\mu,d)$-generic.  However, this does not seem to be an integral
part of the theory.  In light of this and Remark~\ref{rmk:basic}~(2) we ask:

\begin{question}
\label{que:hyperbolic}
Is the assumption that $M$ is $(\mu,d)$-generic necessary? 
\end{question}

\bigskip

\noindent The three ingredients necessary for the proof of Theorem~\ref{thm:hyperbolic} 
are Theorem~\ref{thm:jt}, the work of Rieck and Sedgwick, and Theorem~\ref{thm},
which represents the bulk of the work in this paper.  In this theorem we allow a 
flexible definition of triangulation, which we call generalized triangulation. 
See Definition~\ref{dfn:triangulation} and Lemma~\ref{lem:ExistTriangulations}
for 
existence.

\begin{thm} 
\label{thm}
Let $M$ be a compact orientable connected 3-manifold and 
$\s$ a Heegaard surface for $M$.
Suppose that for some 
(possibly empty or disconnected) compact
surface $K \subset \partial M$, $M \setminus K$ admits a generalized
triangulation with $\tm$ generalized tetrahedra.    

If $g(\s) \geq \bdd$, then $d(\s) \leq 2$.
\end{thm} 

\begin{rmkk}
\begin{enumerate}
	\item Theorem~\ref{thm} generalizes S. Schleimer's 
	\cite[Theorem~11.1]{schleimer}, where it was shown that if $M$ is a 
	closed manifold and  $g(\s) \geq 2^{2^{16}\tm^2}$, then $d(\s) \leq 2$.
	\item Theorem~\ref{thm} implies that 
	for every manifold $M$, there is $g_{M} \geq 0$, so that if $\s \subset M$
	is a Heegaard surface of genus at least $g_{M}$, then $d(\s) \leq 2$;
	this also follows from \cite[Theorem~11.1]{schleimer}.
\end{enumerate}	
\end{rmkk}	


\noindent {\bf Outline.}  In Section~\ref{sec:hyperbolic},
we show how Theorem~\ref{thm:hyperbolic} follows from Theorem~\ref{thm}.
In Section~\ref{sec:questions} we explain our
perspective of Theorem~\ref{thm} and list open questions related to 
it.  In Section~\ref{sec:prelims} we explain a few preliminaries.  The work begins in
Section~\ref{sec:X}, where we take a strongly irreducible Heegaard surface of
genus at least $\bdd$, color it, and analyze the coloring; the climax of
Section~\ref{sec:X} is Proposition~\ref{pro:X}, where we prove existence of a
pair of pants with certain useful properties.  Finally, in Section~\ref{sec:proof} we
prove Theorem~\ref{thm}. 

\medskip

\noindent{\bf Acknowledgment.}  We thank Cameron Gordon, Marc Lackenby, 
Kimihiko Motegi, and
Saul Schleimer for interesting conversations and correspondence.  
We thank the anonymous referees for a careful reading of the paper
and insightful comments that added to the content and improved the presentation of this work.
The second named author: this work was carried out while I was visiting OCAMI at Osaka
City University and Nara Women's University. I thank Professor Akio Kawauchi
and OCAMI, and Professor Tsuyoshi Kobayashi and the  math department of Nara
Women's University for the hospitality I enjoyed during those visits.

\section{Proof of Thoerem~\ref{thm:hyperbolic}}
\label{sec:hyperbolic}

We first show how Theorem~\ref{thm:hyperbolic} follows from Theorem~\ref{thm}.
Fix the notation of Theorem~\ref{thm:hyperbolic}.  Let $\lambda = \frac{1}{A}$,
where $A > 0$ is the constant given in Theorem~\ref{thm:jt}.
By Remark~\ref{rmk:jt}, for any complete finite volume  hyperbolic 3-manifold $M$, 
$N_{d}(M_{\geq \mu})$ can
be triangulated using at most $\lambda \vol[M]$ tetrahedra.

Set $\Lambda = 76 \lambda + 29$.   Let $\s \subset M$ be a Heegaard surface 
of genus $g(\s) \geq \Lambda \vol[M]$.  
Using the definition of $\Lambda$ and the fact that $\vol[M] > .9$ (Gabai, Meyerhoff, and Milley \cite{gmm})
we get:
\begin{eqnarray*}
g(\s) &\geq& \Lambda \vol[M] \\
&=& (76 \lambda + 29) \vol[M] \\
&=& 76 \lambda \vol[M]  + 29 \vol[M] \\
&>& 76 \lambda \vol[M]  + 26. 
\end{eqnarray*}

By assumption, $M$ is a $(\mu,d)$-generic, that is, $M$ is obtained from $N_{d}(M_{\geq \mu})$
by a generic Dehn filling (recall 
Definition~\ref{dfn:generic}).  Hence, after isotopy
if necessary, $\s$ is a Heegaard surface for $N_{d}(M_{\geq \mu})$.  By Remark~\ref{rmk:jt},
$N_{d}(M_{\geq \mu})$ can be triangulated using  $t \leq \lambda \vol[M]$ tetrahedra.  
We see that $g(\s) > 76\lambda \vol[M] + 26 \geq 76 t + 26$, 
and by Theorem~\ref{thm} (applied to $\s$ as a Heegaard
 surface of $N_{d}(M_{\geq \mu})$), 
$d(\s) \leq 2$.  It is elementary to see that distance never increases under Dehn 
filling, and we conclude that $\s \subset M$ is a Heegaard surface of distance at most 2, 
completing the proof of Theorem~\ref{thm:hyperbolic}.

\section{Open Questions}
\label{sec:questions}

Theorem~\ref{thm} is a constraint on the distance of surfaces of genus $\bdd$
or more.  There are other constraints on the distance known, and by far 
the most important is  Casson and Gordon's theorem \cite{cg} that
says that no Heegaard surface of an irreducible, non-Haken 3-manifold has
distance exactly one.  Other examples include W. Haken's theorem that says that any
Heegaard surface of a reducible 3-manifold has distance zero, and T. Li's theorem
\cite{li2} that says that a non-Haken 3-manifold admits only
finitely many Heegaard surfaces of positive distance.
Another constraint is~\cite[Corollary~3.5]{schtom}, where
M. Scharlemann and M. Tomova prove that if $\s_1$ and $\s_2$ are
non isotopic Heegaard surfaces of a closed manifold so that 
$d(\s_2) > 2 g(\s_1)$, then $d(\s_1) = 0$ (in fact, they show that $\s_1$ is obtained from
$\s_{2}$ by stabilization).

On the positive side, all but finitely many
of the surfaces constructed by Casson and Gordon have distance exactly two 
(Casson and Gordon's work show that the distance is at least 2 and
Theorem~\ref{thm:hyperbolic} provides a new proof that the distance is at most 2).
Hempel \cite{hempel}, using a construction of Kobayashi~\cite{KobayashiHeights}, 
shows that for any $g \geq 2$
there exists a sequence of 3-manifolds $M_n$ and Heegaard splittings $\s_n$ for
$M_n$, so that $g(\s_n) = g$ and $\lim_{n\to \infty} d(\s_n) = \infty$.
T. Evans \cite{evans} improved this by constructing, given $g \geq 2$ and $d
\geq 0$, a Heegaard splitting of genus $g$ with distance at least $d$.
Recently, Qiu, Zou, and Guo~\cite{QiuZouGuo} and, independently,
Ido, Jang and Kobayashi~\cite{IJK}, constructed, given 
$g \geq 2$ and $d \geq 1$, a compact manifold with Heegaard splitting of 
genus $g$ and distance exactly $d$. 
In~\cite{yoshi} Yoshizawa shows that when $d$ is even, a Heegaard 
splitting of distance exactly $d$ can be obtained by applying high powers of 
a single Dehn twist.

However, the answers to the following questions are not known in general:

\begin{questions}
\label{questions1}  
  \begin{enumerate}
  \item Given $g_i \geq 2$ and $d_i > 0$ ($i=1,2$), does there exist a
  3-manifold admitting distinct Heegaard surfaces $\s_1$, $\s_2$, so that
  $g(\s_i) = g_i$ and $d(\s_i) = d_i$?
  \item Given $d_i > 0$ ($i=1,2$), does there exist a 3-manifold admitting
  distinct Heegaard surfaces $\s_1$, $\s_2$, so that $d(\s_i) = d_i$?
  \end{enumerate}
\end{questions}

Questions~(1) and~(2) above can naturally be generalized to more that
two surfaces by setting $i=1,\dots,n$, for some chosen $n$.
The word ``distinct'' in the questions above can be
interpreted as ``distinct up to isotopy'' or ``distinct up to homeomorphism'';
both yield interesting questions.  

The answer for Question~\ref{questions1}~(2) is known only in the following
cases:

\begin{itemize}
\item $d_1 = d_2 = 2$: As mentioned above, there are examples of Casson and
    Gordon of 3-manifolds admitting infinitely many Heegaard surfaces of
    unbounded genera and of distance invariant two.  Other examples follow from
    S. Beiler and Y. Moriah~\cite{BleilerMoriah} (see also
    K. Morimoto and M. Sakuma~\cite{MorimotoSakuma}).  They show that there
    exist 2-bridge knots $K$ admitting more than one minimal genus Heegaard
    surface (up to homeomorphism).  Let $\s$ be one of these surfaces.  It
    is easy to see that $d(\s)=2$: first, since $g(\s)=2$, it is easy to see that
    $d(\s) \geq 2$.  Next, $\s$ is constructed by viewing $K$ as a torus
    1-bridge knot (that is, there exists a genus 1 Heegaard splitting 
    $T_1 \cup T_{2}$ so that $K$ intersects each $T_i$ in a single unknotted arc)
    and tubing once.  Meridian disks for $T_i$ which are disjoint from $K$
    and the tube, are also disjoint from the core of the tube,  showing that
    $d(\s) \leq 2$. 
\item $d_1 = d_2 = 1$:  
    Let $S$ be a 4-punctured sphere and $M = S \times S^1$.
    J. Schultens \cite{schultens} showed that $g(M) = 3$.  
    We note that $M$ admits two
    minimal genus Heegaard splittings, say $\s_1$ and $\s_2$, such that $\s_1$
    is obtained by tubing three boundary parallel tori, and $\s_2$ is
    obtained by tubing two boundary parallel tori, with an
    extra tube that wraps around a third boundary component.  Since $\s_1$ and
    $\s_2$ induce boundary partitions with distinct numbers of components, they
    are distinct up to homeomorphism.  By construction, $d(\s_1) = d(\s_2) = 1$.
\item $d_1 = 1, \ d_2 = 2$: In \cite{KRSIWR} the authors constructed a 3-manifold
    $M$ admitting minimal genus Heegaard splittings $\s_1$, $\s_2$, with $d(\s_2) = 2$ and
    $d(\s_1)  = 1$.  In this example, $g(M) = g(\s_1) = g(\s_2) = 3$.
\item $d_{1} =  d_{2} = 3$: Scharlemann~\cite{MR2823137}, based on a preprint by
Berge~\cite{berge}, shows that there exists a closed manifold $M$ admitting two 
Heegaard splittings $\s_{1}$ and $\s_{2}$, distinct up-to homeomorphism,
so that $g(M) = g(\s_{1}) = g(\s_{2})=2$ and $d(\s_{1}) = d(\s_{2})=3$.
\end{itemize}

We see that much is known when $d_1$, $d_2 \leq 3$.  By contrast, the answers 
to the following basic questions are unknown:

\begin{questions}
\label{questions2}  
  \begin{enumerate}
  \item Does there exist a 3-manifold admitting two (or more) distinct Heegaard surfaces
  with distance four or more?
  \item Does there exist a 3-manifold admitting a Heegaard surface of distance 
  three or more that is not of minimal genus?
  \end{enumerate}
\end{questions}

\section{Preliminaries}
\label{sec:prelims}

By {\it manifold} we mean compact, connected, orientable 3-manifold.  We assume familiarity with
the basic notions of 3-manifold topology (see, for example, \cite{hempel-book}
or \cite{jaco}) and the basic facts about Heegaard splittings (see, for
example, \cite{scharlemann} or \cite{sss}).   
We use the notation $N(\ )$ for open normal neighborhood,
$\del$ for boundary, and $| \ \ |$
for the number of components.  We define:

\begin{dfns}
\label{dfn:triangulation}
\begin{enumerate}
\item Let $T$ be a tetrahedron.  A {\it generalized tetrahedron} is obtained
  by fixing two disjoint sets of vertices of $T$, denoted $V_1$, $V_2$, and then 
  removing $V_1$ and  truncating $V_2$; that is, a generalized tetrahedron $T'$ has the form 
$T' = T \setminus (V_1 \cup N(V_2))$.   
  $T'$ has exactly four faces (resp. exactly
  six edges, at most four vertices), which are the intersection of the faces
  (resp. edges, vertices) of $T$ with $T'$.  In particular, the components of
  $\del N(V_2) \cap T'$ are not considered faces.  Important special cases are when
  $V_2 = \emptyset$, then $T'$ is called {\it semi-ideal}, and when $V_1$
  consists of all four vertices, then $T'$ is called {\it ideal}.
\item A {\it generalized triangulation} is obtained by gluing together
  finitely many generalized  tetrahedra, where the gluings are done by
  identifying faces, edges, and vertices.   Self-gluings (that is,
  gluing a tetrahedron to itself) are allowed, as are multiple gluings (that is, gluing
  two tetrahedra along more than one face).   We refer the reader to
  \cite{hatcher} for a detailed description in the special case when only
  tetrahedra are used, known there as $\Delta$ complexes.  If all the
  generalized tetrahedra are ideal (resp. semi ideal), then the generalized
  triangulation is called an {\it ideal} (resp. {\it semi  ideal})
  triangulation.  If the quotient obtained is homeomorphic to a given manifold
  $M$ it is said to be a {\it generalized triangulation of $M$.}
\end{enumerate}
\end{dfns}

We refer the reader to, for example, \cite[Section~2]{LackenbyAlgorithm} for a
detailed discussion of generalized tetrahedra.  It is well known that a very
large class of 3-manifolds admits generalized triangulations, including
all compact 3-manifolds.  We outline the proof here.   
Let $W$ be a compact manifold and $K_i \subset \del W$ ($i=1,\dots,n$) a disjoint, closed, 
connected subsurfaces.  By crushing each $K_i$ to a point $p_i$, we obtain a
3-complex $X$.  We can triangulate $X$ so that each $p_i$ is a vertex of the
triangulation.  Removing $p_i$ we obtain a semi-ideal triangulation of $N
\setminus (\cup_i K_i)$.  We conclude that 
(with $K$ corresponding to $\cup_{i} K_{i}$):

\begin{lem}
\label{lem:ExistTriangulations}  
Let $M$ be a compact manifold and $K \subset \del M$ a (not necessarily connected) 
closed subsurface.  Then  $M \setminus K$ admits a generalized triangulation.  
\end{lem}

\bigskip\noindent
In \cite{hempel} Hempel defined the {\it distance} of a Heegaard splitting:

\begin{dfn}
\label{dfn:distance} 
Let $V_1 \cup_\s V_2$ be a Heegaard splitting and $d \geq 0$ an integer.  We
say that the distance of $\s$ is $d$, denoted by $d(\s) = d$, if $d$ is the
smallest integer so that there exist meridian disks $D_1 \subset V_1$ and $D_2
\subset V_2$, and 
essential curves $\alpha_i \subset \s$ ($i=0,\dots,d$), so that $\alpha_0 =
\del D_1$, $\alpha_d = \del D_2$, and $\alpha_{i-1} \cap \alpha_i
= \emptyset$ ( for $1 \leq i \leq d$).
\end{dfn}
%

The following lemma is easy and well known (see, for example
\cite[Remark~2.6]{schleimer}): 

\begin{lem}
\label{lem:distance}
Let $V_1 \cup_\s V_2$ be a Heegaard splitting.  Suppose that one of the
following holds:
\begin{enumerate}
\item for $i=1,2$, there exists a properly embedded, non-boundary parallel
  annulus $A_i \subset V_i$, and there exists an essential curve $\alpha
  \subset \s$ so that $\alpha \subset A_1 \cap A_2$ (that is to say, $A_{1}$ and 
  $A_{2}$ have an  essential common boundary component), or:
\item there exist a meridian disk $D_1 \subset V_1$ and a properly embedded
  non-boundary parallel annulus $A_2 \subset V_2$, 
  so that $D_{1}$ is disjoint from at least one component of $\del A_{2}$ 
  that is essential in $\s$.
\end{enumerate}
Then $d(\s) \leq 2$.
\end{lem}

\section{Coloring $\s$ and constructing the pair of pants $X$}
\label{sec:X}

Fix $M$ as in the statement of Theorem~\ref{thm} and let $V_1 \cup_\s V_2$ be
a Heegaard splitting for $M$ with $g(\s) \geq \bdd$.
Let $\T$ be a generalized triangulation of 
$M \setminus K$ (where $K \subset \del M$ is a closed subsurface)
with $\tm$ generalized tetrahedra.   

If $\s$ weakly  reduces, then $d(\s) \leq 1$; we assume as we may
that $\s$ is strongly irreducible. Rubinstein \cite{rubinstein} (see also 
Stocking \cite{stocking} and Lackenby~\cite{lackenby}\cite{LackenbyAlgorithm} when 
$M$ is not closed) show that $\s$ is isotopic to an {\it almost 
normal} surface, that is, after isotopy the intersection of $\s$ with the
generalized tetrahedra of $\T$ consists of {\it normal  faces}, of which there are two
types: 
\begin{enumerate}
\item normal disks (normal triangles and normal quadrilaterals)
\item an exceptional component, which is either an octagonal disk or an 
annulus obtained by tubing together two normal disks; at most one normal 
face of $\s$ is an exceptional component.
\end{enumerate}

Let $N$ be a regular neighborhood of $\mathcal{T}^{(1)}$, the 1-skeleton of
$\mathcal{T}$.  For each $v \in \mathcal{T}^{(1)}  \cap \s$, let $D_v$
be the component of $\s \cap N$ containing $v$.  Then $D_v$ is a disk properly
embedded in $N$, called the {\it vertex disk corresponding to} $v$.  Let
$\widehat{F}$ be a normal face contained in a generalized tetrahedron $T$.  Then $F =
\widehat{F} \setminus \mbox{int}N$ is obtained from $\widehat{F}$ by removing
a neighborhood of the vertices of $\widehat{F}$.  $F$ is called a {\it
truncated normal face}.  For the remainder of this paper, by a {\it face} we
mean a truncated normal face or a vertex disk.  

\begin{rmkk}
\label{rmk:3valent}
The union of the boundaries of the faces forms a 3-valent graph in $\s$.      
\end{rmkk}

Let $v,\ v' \in \mathcal{T}^{(1)} \cap \s$ 
be two vertices and $D_v$, $D_{v'}$
the corresponding vertex disks.  Then $D_v$ and $D_{v'}$ are called {\it
$I$-adjacent}  if $v$ and $v'$ are contained in the same edge $e \in
\mathcal{T}^{(1)}$ and $v$ is adjacent to $v'$ along $e$.
Note that $D_v$ is $I$-adjacent to $D_{v'}$
if and only if $v$ and $v'$ are contained in the same edge $e \in
\mathcal{T}^{(1)}$ and  there exists an $I$-bundle over $D^{2}$
with total space $Q \subset N$, so that
$\del Q \setminus (D_v \cup D_{v'}) \subset \del N$, $Q \cap \s = D_v \cup
D_{v'}$, and $D_v \cup D_{v'}$ is the associated $\del I$-bundle.  

Let $F$ and $F'$ be truncated normal faces.  Then $F$ and $F'$ are called {\it
$I$-adjacent}  if the corresponding normal
faces are parallel and there is no normal face between the two.  Note
that $F$ and $F'$ are $I$-adjacent if and only if they
are contained in the same generalized tetrahedron $T$,
and  there exists an $I$-bundle with total space
$Q \subset T \setminus \mbox{int}N$, so that  
$\del Q \setminus (F \cup F') \subset \del (T \setminus \mbox{int}N)$
and is disjoint from the vertices, truncated vertices, and missing vertices, 
 $Q \cap \s = F \cup F'$, and $F \cup F'$ is the associated $\del I$-bundle.  

Clearly $I$-adjacency is 
symmetric but 
not, in general, transitive.  The equivalence relation generated by 
$I$-adjacency is called {\it $I$-equivalence}, and its
equivalence classes are  called {\it $I$-equivalent families.}
For example, suppose that a tetrahedron contains four quadrilaterals
and denote the corresponding truncated normal faces by
$q_{1},q_{2},q_{3},q_{4}$ (listed in order).  If there is a truncated exceptional
piece between $q_{2}$ and $q_{3}$, then the truncated quadrilaterals 
form exactly two $I$-equivalent
families: $\{q_{1},q_{2}\}$ and $\{q_{3},q_{4}\}$.

Let $\mathcal{F}$ be an $I$-equivalent family.  Then $I$-adjacency 
induces a linear ordering on the faces in $\mathcal{F}$, ordered as
$F_1,\dots,F_n$, so that  $F_i$ is $I$-adjacent to $F_{i+1}$
($i=1,\dots,n-1$).  This order is unique up-to  reversing.  We color the faces
of $\mathcal{F}$ as follows:
\begin{enumerate}
\item $F_1$, $F_{2}$, $F_{n-1}$, and $F_n$ are colored red.
\item If $n \geq 5$, then $F_3,\dots,F_{n-2}$ are colored blue and
  yellow alternately.  Note that this leaves us the freedom to exchange the
  blue and yellow colors of the faces of $\mathcal{F}$.
\end{enumerate}

\begin{rmk}
For most of our work it suffices to color red $F_{1}$ and $F_{n}$.  
We need to color $F_{2}$ and $F_{n-1}$ red as well for the last case of the proof of 
Theorem~\ref{thm}, where a further refinement of the colors
will be given.
\end{rmk}

By construction, any yellow or blue face is  $I$-adjacent to two distinct faces.

\begin{rmkk}
\label{rmk:RedDisk}
Let $D_v$ be a red vertex disk.  By construction, $D_v$ is outermost or next
to outermost along an edge of $\mathcal{T}^{(1)}$.  Therefore all the truncated
normal faces that intersect $D_v$ are red as well.      
\end{rmkk}

\begin{lem}
\label{lem:NumberOfRed}
Let $\rt$ denote the number of the red truncated triangles and $\rqq$ the
number of the red truncated quadrilaterals.  Then one of the following holds:
\begin{enumerate}
\item $\rt \leq 16\tm$ and $\rqq \leq 4\tm + 4$.
\item $\rt \leq 16\tm + 4$ and $\rqq \leq 4\tm$.
\end{enumerate}
\end{lem}

\begin{proof}
A generalized tetrahedron not containing the exceptional component admits at most four
$I$-equivalent families of truncated triangles and one $I$-equivalent  family
of truncated quadrilaterals.   If there is an exceptional component, the
generalized tetrahedron containing it admits at most five $I$-equivalent families of
truncated triangles and  one $I$-equivalent family of truncated
quadrilaterals, or at most four $I$-equivalent families of truncated triangles
and two $I$-equivalent families of truncated quadrilaterals.  Each family
contains at most four red faces.  The lemma follows.
\end{proof}

Let $B$ (resp. $Y$, $R$) denote the union of the blue (resp. yellow, red)
faces; note that faces are closed, 
so $R$, $Y$, and $B$ are compact and may intersect along their boundaries.  
By Remark~\ref{rmk:3valent}, $B$, $Y$, $R$, and $B \cup Y$ are
subsurfaces of $\s$.

\begin{lem}
\label{lem:chi}
$\chi(B \cup Y) \leq -(108 \tm + 38)$.
\end{lem}

\begin{proof}
We first show that $\chi(R) \geq -(44\tm+12)$; for that, 
we order the red faces as $F_{0}$, $F_{1},\dots,F_{k}$, $F_{k+1},\dots,F_{n}$ (for some $k$, $n$)
so that $F_{0}$ is the exceptional piece (if there is one, $F_{0} = \emptyset$ otherwise),
 $F_{1},\dots,F_{k}$ are the red truncated normal faces, and $F_{k+1},\dots,F_{n}$
are red vertex disks.  Note that $\chi(F_{0}) = 0$ or $\chi(F_{0}) = 1$, so the
worst case scenario is 0.   By Remark~\ref{rmk:3valent}, 
for $0 \leq i \leq k$, the possibilities
for  $F_{i} \cap (\cup_{j=1}^{i-1} F_{j})$ are: $\emptyset$, $S^{1}$,
or a number of segments, each homeomorphic to $I$.   Since a truncated  
normal triangle (respectively quadrilateral)
is a hexagon (respectively octagon), the number of segments is at most 3 (respectively 4).
We see that 
$$\chi(\cup_{j=1}^{i} F_{j}) \geq \chi(\cup_{j=1}^{i-1} F_{j}) -2$$
when $F_{i}$ is a truncated normal triangle and 
$$\chi(\cup_{j=1}^{i} F_{j}) \geq \chi(\cup_{j=1}^{i-1} F_{j}) -3$$
when $F_{i}$ is a truncated normal quadrilateral. 
By Remark~\ref{rmk:RedDisk}, for $i \geq k+1$, $F_{i}$
caps a hole of $\cup_{j=1}^{i-1} F_{i}$; hence 
$$\chi(\cup_{j=1}^{i} F_{j}) = \chi(\cup_{j=1}^{i-1} F_{j}) +1$$
in that case.  Recall
that $\rt$ and $\rqq$ were 
defined and bounded in Lemma~\ref{lem:NumberOfRed}.  Adding the contributions
of the exceptional component (at worst $0$), the triangles (at worst $-2\rt$),
the quadrilaterals (at worst $-3\rqq$), and ignoring the positive contribution
of the vertex disks, Lemma~\ref{lem:NumberOfRed} gives:
\begin{eqnarray*}
  \chi(R)  &\geq& 0 - 2\rt -3 \rqq \\       %
          &\geq& 0 - 2 (16\tm) -3 (4\tm+4) \\   %
          &=& -(44\tm+12).        %
\end{eqnarray*}

%
%
Since $R$ and $B \cup Y$ are subsurfaces, $\s = R \cup (B \cup Y)$,
and $R \cap (B \cup Y) = \del R = \del (B \cup Y)$ consists of circles,
we have that $\chi(B \cup Y) = \chi(\s) - \chi(R)$.  By assumption $g(\s) \geq \bdd$, or
equivalently $\chi(\s) \leq 2-2(\bdd)$.  Hence:
\begin{eqnarray*}
  \chi(B \cup Y) &=& \chi(\s) - \chi(R) \\ &\leq& [2-2(\bdd)] + [44\tm+12] \\
         &=& -(108 \tm + 38).
\end{eqnarray*}
\end{proof}

\begin{lem}
\label{lem:del}
$|\partial (B \cup Y)| \leq 44\tm + 14$.
\end{lem}

\begin{proof}
By construction $\partial (B \cup Y) = \partial R$.  
Bounding $|\del R|$ is similar to the proof of the previous lemma and we only paraphrase
it here: we order
the red faces as $F_{0},\dots,F_{n}$ as in the proof of the previous lemma.
It is easy to see that $|\del F_{0}|$ is at most 2, and (similar to the Euler
characteristic count on the previous lemma) for $1 \leq i \leq k$, 
$$|\del (\cup_{j=1}^{i} F_{j})| \leq |\del (\cup_{j=1}^{i-1} F_{j})| +2$$
when $F_{i}$ is a truncated normal triangle, and
$$|\del (\cup_{j=1}^{i} F_{j})| \leq |\del (\cup_{j=1}^{i-1} F_{j})| +3$$
when $F_{i}$ is a truncated normal quadrilateral.
By Remark~\ref{rmk:RedDisk}, for $i \geq k+1$,
$$|\del (\cup_{j=1}^{i} F_{j})| \leq |\del (\cup_{j=1}^{i-1} F_{j})| -1.$$
Adding up the contributions of the truncated normal faces
and ignoring the negative contribution of the vertex disks, 
Lemma~\ref{lem:NumberOfRed} gives:
\begin{eqnarray*}
  |\partial (B \cup Y)| &=& |\partial R| \\  %
               &\leq& 2 + 2\rt + 3 \rqq \\  %
               &\leq& 2 + 2(16\tm) + 3(4\tm+4) \\ %
               &=& 44\tm + 14.  %
\end{eqnarray*}
\end{proof}

By Remark~\ref{rmk:3valent}, $B \cap Y$ is a compact 1-manifold properly embedded in
$B\cup Y$.  Let $\Gamma  \subset B \cup Y$ be the union of the arc components
of $B \cap Y$.  Endpoints of $\Gamma$ are the vertices of $\s$ where red, blue,
and yellow faces meet.  By Remark~\ref{rmk:RedDisk} around any vertex of $\Sigma$ that is on the boundary
of a red vertex disk all the colors are red; therefore the vertex disk at an endpoints of $\Gamma$ 
is yellow or blue.

Let $\mathcal{V}$ be the set of vertices of red truncated
normal faces.  We subdivide $\mathcal{V}$ into 3 disjoint sets as follows:
$\mathcal{V}_0$ are vertices that are on the boundary of at least two red faces; 
$\mathcal{V}_+$ are vertices that are on the boundary of
three faces so that one is red, one is yellow, and one is blue; 
$\mathcal{V}_-$ are vertices that are on the boundary of three faces
so that one is red and two are yellow, or one is red and two are blue.   
By construction, $\mathcal{V}_+$ is exactly the set of endpoints of $\Gamma$. 

By construction, at every vertex exactly one face is a vertex disk.
We exchange the colors of the blue vertex disks with the colors of the yellow
vertex disks; let $R'$, $B'$, $Y'$ and  $\Gamma'$ be defined as above, using
the new  coloring.  By Remark~\ref{rmk:RedDisk}, $\mathcal{V}_-$ is exactly the
set of endpoints of $\Gamma'$ (we emphasize that $\mathcal{V}_-$ is the set of
vertices defined above using the original coloring). Hence, by exchanging
colors if necessary, we may assume that the number of endpoints of $\Gamma$ is
at most $\frac{1}{2}|\mathcal{V}|$.    Since every arc of $\Gamma$ has two
distinct endpoints and $\Gamma$ has at most $\frac{1}{2}|\mathcal{V}|$
endpoints, we get that $|\Gamma| \leq \frac{1}{4}|\mathcal{V}|$.

There are at most 16
vertices in $\mathcal{V}$ from the truncated exceptional component, at most 6
from each truncated red triangle, and at most 8 from each truncated red
quadrilateral.  By Lemma~\ref{lem:NumberOfRed} we get:
\begin{eqnarray*}
  |\mathcal{V}| &\leq&  16 + 6\rt + 8 \rqq  \\  %
               &\leq& 16 + 6 (16\tm) + 8(4\tm + 4)  \\  %
               &\leq& 128\tm + 48.%
\end{eqnarray*}
Hence:
$$|\Gamma| \leq \frac{1}{4} |\mathcal{V}| \leq 32 \tm + 12.$$

Let $F_1,\dots,F_k$ be the components of $B \cup Y$ cut open along $\Gamma$
(note that $F_{1},\dots,F_{k}$ are not, in general, faces).
Cutting along $\Gamma$ increases the Euler characteristic by $|\Gamma|$ and 
increases the number of boundary components by at most $|\Gamma|$.  Using
Lemma~\ref{lem:chi} we get:
\begin{eqnarray*}
  \s_{i=1}^k \chi(F_i)  &=& \chi(\cup_{i=1}^k F_i)   \\  %
               &=&   \chi(B \cup Y) + |\Gamma|\\  %
               &\leq& -(108\tm + 38) +  (32 \tm + 12) \\%
               &=& -(76\tm + 26).
\end{eqnarray*}
And using Lemma~\ref{lem:del} we get:
\begin{eqnarray*}
  \s_{i=1}^k |\del F_i|  &=& |\del \cup_{i=1}^k F_i|   \\  %
               &\leq&   |\del (B \cup Y)| + |\Gamma|\\  %
               &\leq& (44\tm + 14) +  (32 \tm + 12) \\%
               &=& 76\tm + 26.
\end{eqnarray*}
Combining these inequalities we get:
\begin{equation}
\label{eq}
  \s_{i=1}^k \chi(F_i) \leq -(\s_{i=1}^k |\del F_i|).
\end{equation}

\begin{pro}
\label{pro:X}  
There exists a pair of pants $X \subset \s$ with the following two properties:
\begin{enumerate}
\item Either $X \subset \mbox{int}(B)$ or
$X \subset \mbox{int}(Y)$ (say the former).
\item The components of $\del X$, denoted by $\alpha$, $\beta$, and $\gamma$,
  are essential in $\s$.
\end{enumerate}
\end{pro}

\begin{proof}
By Inequality~(\ref{eq}) above, for some $i$, $\chi(F_i) \leq -|\del F_i|$;
equivalently, $g(F_i) \geq 1$.  Fix such $i$.  By construction, $(B \cap Y)
\cap \mbox{int}(F_i)$ consists of simple closed curves; see
Figure~\ref{fig:F}.  Let $\mathcal{E}$ (resp. $\mathcal{I}$) denote the curves
of $(B \cap Y) \cap \mbox{int}F_i$ that are essential (resp. inessential) in $\s$. 
\begin{figure}[htbp]
     \centerline{  \includegraphics[width=4in]{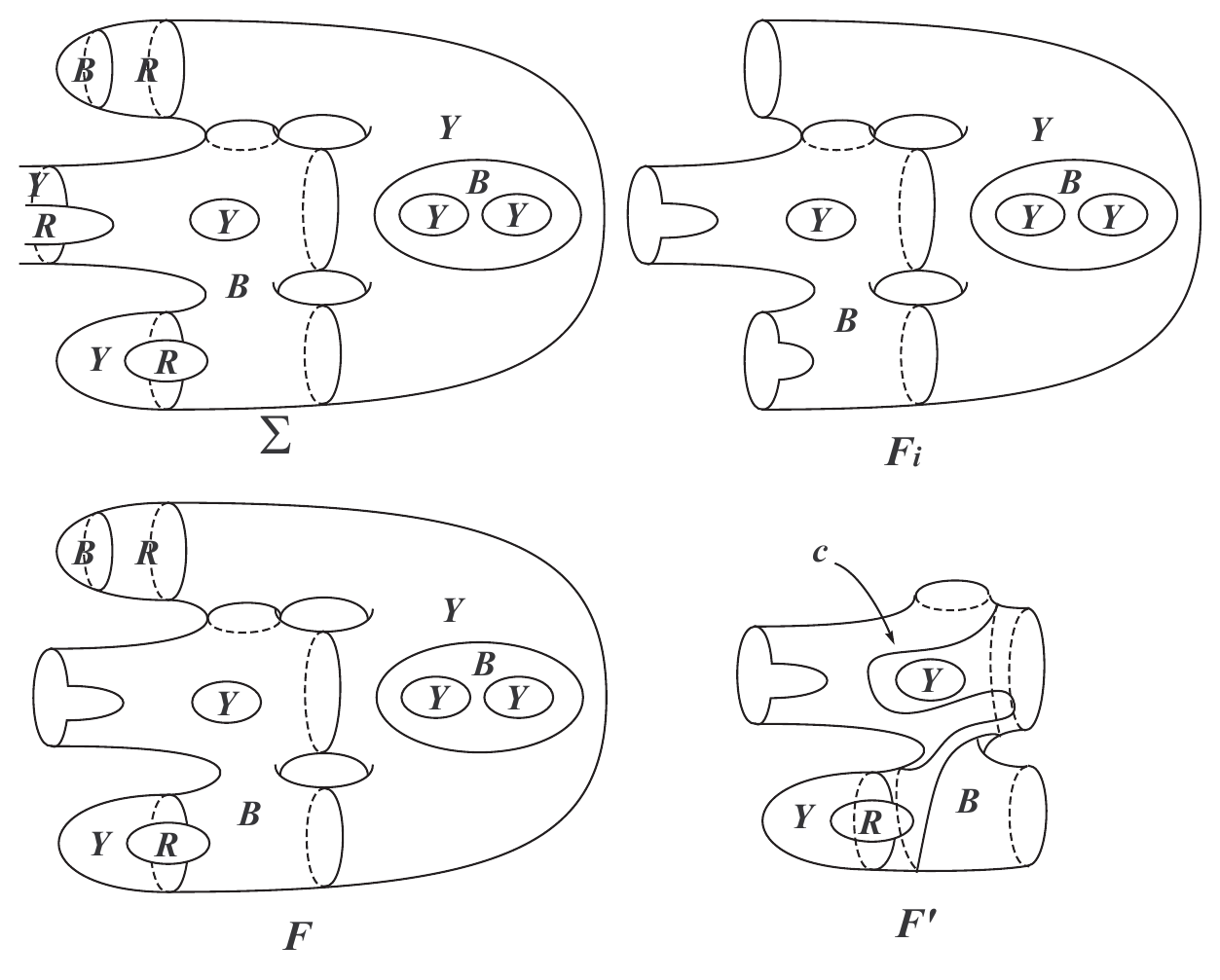}}
  \caption{}
  \label{fig:F}
\end{figure}

Let $\Delta$ be the union of the components of $\mbox{cl}(\s \setminus F_i)$
that are  disks (possibly, $\Delta = \emptyset$).  Let $F = F_i \cup \Delta$.
By construction, every component of $\del F$ is essential in $\s$ (possibly,
$\del F = \emptyset$).  Thus, a closed curve of $F$ is essential in $\s$ if
and only if it is essential in $F$.   
Since $g(F) = g(F_{i}) > 0$, if $\del F \neq \emptyset$ then $\chi(F) < 0$; if,
on the other hand, $\del F = \emptyset$, then $F = \s$
and in particular, $g(F) = g(\s) \geq 76t + 26 >1$;  
we conclude that in either case $\chi(F) < 0$.
Thus some component of
$F$ cut open along $\mathcal{E}$, denoted by $F'$, has $\chi(F') < 0$.  Note that
every curve of $\del F'$ is essential in $\s$.  By construction, $(B \cap Y)
\cap \mbox{int}F' \subset \mathcal{I}$.  Let $\Delta'$ be the union of the
disks bounded by outermost curves of $\mathcal{I} \cap F'$ and the disks
$\Delta \cap F'$.  Note that $\Delta' \subset \mbox{int}F'$ consists of disks,
and $F' \setminus \Delta'$ is entirely blue or yellow; in Figure~\ref{fig:F},
$\Delta'$ consists of two disks, one of each kind.

Assume first that $\del F' \neq \emptyset$.  Let $c \subset F'$ be a curve,
parallel to a component of $\del F'$, that decomposes $F'$ as $A'' \cup_c
F''$, where $A''$ is an annulus.  By isotopy of $c$ in $F'$ we may assume that
$\Delta' \subset A''$.    We see that $F''$ is entirely blue or yellow,
$\chi(F'') = \chi(F') < 0$, and $\del F''$ is essential in $\s$.

Next assume that $\del F' = \emptyset$ (that is, $F' = \s$).  Let $c$ be a
separating, essential curve in $F'$.  By isotopy of $c$ we may assume that
$\Delta'$ is contained in one 
component of $F'$ cut open along $c$.  Let $F''$ be the other component.  We
conclude that in this case too, $F''$ is entirely blue or yellow, $\chi(F'') <
0$, and $\del F''$ is essential in $\s$.

Let $\alpha$, $\beta$, and $\gamma \subset \mbox{int}(F'')$ be three curves 
that are essential in 
$F''$ (and hence in $\s$) and co-bound a pair of pants, denoted by $X$, in
$F''$.  It is easy to see that $X$,  $\alpha$, $\beta$, and $\gamma$  have the properties listed in
Proposition~\ref{pro:X}. 
\end{proof}

Since $X \subset \mbox{int}(B)$ it is on the boundary of the total space of an $I$-bundle in $V_i$
($i=1,2$).  The other component of the associated $\del I$-bundle is a pair of
pants denoted by $X_i$.  The components of $\del X_{i}$ are denoted by $\alpha_i$, $\beta_i$, and
$\gamma_i$ so that $\alpha_{i}$ is parallel to $\alpha$, 
$\beta_{i}$ is parallel to $\beta$, and $\gamma_{i}$ is parallel to $\gamma$.  Since
$X \subset \mbox{int}(B)$, every point of $X_i$ is yellow or red; we conclude that $X
\cap X_i = \emptyset$.  
Hence the $I$-bundle in $V_{i}$ is trivial.
The annulus extended from $\alpha$ to $\alpha_{i}$ 
(resp. $\beta$ to $\beta_{i}$, $\gamma$ to $\gamma_{i}$) in $V_i$ is denoted by $A_i$ (resp. $B_i$, $C_i$).  
By construction, these annuli are embedded.
Note that $X_{1} \cap X_{2} \neq \emptyset$ is possible.
%
%
%
%

\begin{lem}
\label{lem:assu1}
One of the following holds:
\begin{enumerate}
\item After renaming if necessary,
$A_1 \subset V_1$ and $B_2 \subset
  V_2$ are not boundary parallel, and $A_2 \subset V_2$, $B_1 \subset V_1$,
  and $C_1 \subset V_1$ are boundary parallel.
\item $d(\s) \leq 2$.  
\end{enumerate}
\end{lem}

\begin{proof}
We claim that one of $A_i$, $B_i$ or $C_i$ is not boundary
parallel in $V_i$ ($i=1,2$).  Suppose, for a contradiction, that $A_i$, $B_i$,
$C_i$ are  all boundary parallel.   Let $\widetilde{A}_{i} \subset V_{i}$ 
be the annulus that $A_{i}$ is parallel to.  Since $X$ is an essential
pair of pants it is not contained in $\widetilde{A}_{i}$; it is
easy to see that the intersection of the region of parallelism between $A_{i}$
and $\widetilde{A}_{i}$ and the trivial $I$-bundle in $V_{i}$ is exactly
$A_{i}$; similarly we treat $B_{i}$ and $C_{i}$.  We see that $V_{i}$ is 
homeomorphic to the trivial $I$-bundle, and hence is a genus 2 handlebody.
This contradicts our assumption that  $g(\s) \geq \bdd > 2$.


Therefore one of $A_1$, $B_1$ or $C_1$ is not boundary parallel, and after
renaming if necessary we may assume it is $A_1$.  We may assume $A_2$ is
boundary parallel, for otherwise $d(\s) \leq 2$ by Lemma~\ref{lem:distance}~(1).  
Similarly, one of $A_2$, $B_2$ or $C_2$ is not boundary parallel,
after renaming if necessary we may assume it is $B_2$, while $B_1$ is boundary 
parallel.  Finally by  Lemma~\ref{lem:distance}~(1) we may assume that $C_1$
or $C_2$ is boundary parallel, say $C_1$.
\end{proof}

\begin{lem}
\label{lem:assu2}  
One of the following holds:
\begin{enumerate} 
\item $\alpha_1$, $\beta_2$ and $\gamma_2$ are essential in $\s$, and 
  $\alpha$ is not isotopic in $\s$ to $\alpha_1$, $\beta$ or $\gamma$. 
\item $d(\s) \leq 2$.  
\end{enumerate}
\end{lem}

\begin{figure}[htbp]
     \centerline{  \includegraphics[width=2.5in]{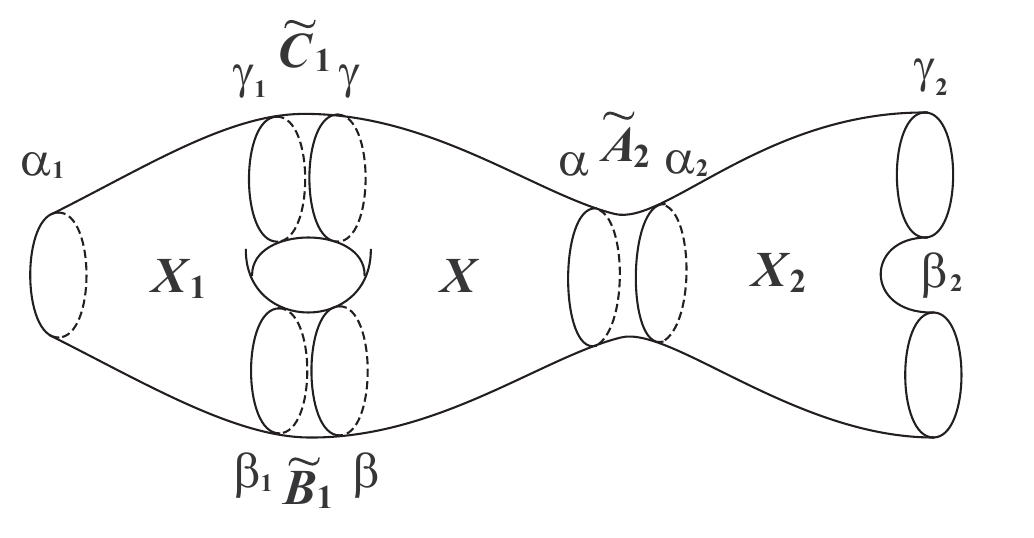}}
  \caption{}
  \label{fig:X}
\end{figure}

\begin{proof}
We may assume that Conclusion~(1) of Lemma~\ref{lem:assu1} holds;
thus $A_{2}$, $B_{1}$ and $C_{1}$ are boundary parallel.
We denote by $\widetilde{A}_2, \widetilde{B}_1, \widetilde{C}_1 \subset \s$ 
the annuli to which $A_{2}$, $B_{1}$, and $C_{1}$ are parallel (respectively).
See Figure~\ref{fig:X}, where $X_1 \cap X_2 = \emptyset$, but this need not be the 
case. 

If $\alpha_1$ is inessential in $\s$, then we may cap $A_1$ off, and after a
small isotopy we obtain a meridian disk $D_1 \subset V_1$ with $\del D_1 = \alpha$.
Using $D_1$ and $B_2$, Lemma~\ref{lem:distance} (2) shows that $d(\s) \leq 2$.
Similarly if $\beta_2$ (resp. $\gamma_2$) is inessential in $\s$ then $\beta$
(resp. $\gamma$) bounds a meridian disk $D_2 \subset V_2$.  Using $D_2$ and $A_1$,
Lemma~\ref{lem:distance} (2) shows that $d(\s) \leq 2$.

If $\alpha$ is isotopic to $\alpha_1$ in $\s$ then either the annulus
connecting the two contains $X$ or $g(\s) = 2$.  The former is impossible
since $X$ is an essential pair of pants and the
latter contradicts the assumption $g(\s) \ge 76\tm + 26 > 2$.

Let $c \subset \s$ be a closed curve 
constructed by pasting together four arcs, the first connecting 
$\beta$ to $\gamma$ in $X$, the second connecting $\gamma$ to 
$\gamma_{1}$ in $\widetilde{C}_{1}$, the third connecting
$\gamma_{1}$ to $\beta_{1}$ in $X_{1}$, and the final arc 
connecting $\beta_{1}$ to $\beta$ in $\widetilde{B}_{1}$.
Since $X \cap X_{1} = \emptyset$, we have
$|c \cap \beta| = |c \cap \gamma| = 1$.
By construction $|c \cap \alpha| = 0$.  Therefore $\alpha$ is not isotopic in $\s$
to either $\beta$ or $\gamma$.
\end{proof}

\section{Proof of Theorem~\ref{thm}}
\label{sec:proof}

With notation as in Section~\ref{sec:X} we assume, as we may by Lemma~\ref{lem:assu1}, that
$A_1$ and $B_2$ are not  boundary parallel and that $A_2$, $B_1$, and $C_1$
are boundary parallel.  We assume, as we may by Lemma~\ref{lem:assu2}, that
$\alpha_1$, $\beta_2$ and $\gamma_2$ are essential in $\s$ and $\alpha$ is not
isotopic in $\s$ to $\alpha_1$, $\beta$ or $\gamma$.

The proof is divided into the following two cases:

\medskip

\noindent {\bf Case One. $\alpha_1$ can be isotoped to be disjoint from
  $X_2$.} 
Let  $\widetilde{A}_2$, $\widetilde{B}_1$, and $\widetilde{C}_1$ be as in
Lemma~\ref{lem:assu2}. 
Let $T \subset \s$ be the twice punctured torus $X \cup \widetilde{B}_1 \cup
\widetilde{C}_1 \cup X_1$.  Isotope $\alpha_1$ so that $\alpha_1 \cap X_2 =
\emptyset$.  After this isotopy, $X_2 \cap (\alpha_1 \cup X) = \emptyset$.
Hence either $X_2 \subset (X_1 \cup \widetilde{B}_1 \cup \widetilde{C}_1)$ or
$X_2 \cap T = \emptyset$.  In the former case, $\alpha_2 \subset (X_1 \cup
\widetilde{B}_1 \cup \widetilde{C}_1)$.  Since $\alpha$ is isotopic to
$\alpha_2$ in $\s$, $\alpha$ is isotopic into $X_1 \cup \widetilde{B}_1 \cup
\widetilde{C}_1$.  By Proposition~\ref{pro:X}(2) $\alpha$ is essential,  and
hence $\alpha$ is isotopic to a component of $\del(X_1 \cup \widetilde{B}_1
\cup \widetilde{C}_1) = \alpha_1 \cup \beta \cup \gamma$, contradicting
our assumptions.  

Hence we may assume that $X_2 \cap T = \emptyset$.  Let $D_1 \subset V_1$ be a meridian disk
obtained by compressing or boundary compressing $A_1$.  After a small isotopy
we may assume that $\del D_1 \cap \del A_1 = \del D_1 \cap (\alpha \cup
\alpha_1) = \emptyset$, and hence either $\del D_1 \subset T$ (hence $\del D_1
\cap \beta_2 = \emptyset$) or  $\del D_1 \cap T = \emptyset$ (hence $\del D_1
\cap \beta = \emptyset$).  Thus $D_{1}$ is disjoint from at least one component
of $\del B_{2}$; by Lemma~\ref{lem:distance}~(2), $d(\s) \leq 2$, proving
Theorem~\ref{thm} in Case One.

\bigskip


\bigskip
\noindent Before proceeding to Case Two we refine our colorings.  
Let $\mathcal{F}$ be an $I$-equivalent family of faces, ordered as $F_{1},\dots,F_{n}$
so that $F_{i}$ is $I$-adjacent to $F_{i+1}$ ($i=1,\dots,n-1$).  Then the red faces
are $F_{1}$, $F_{2}$, $F_{n-1}$, and $F_{n}$.  We color $F_{1}$ and $F_{n}$
dark red.  If $n \ge 3$ we color $F_{2}$ and $F_{n-1}$ light red.

Clearly, a face is  $I$-adjacent to two distinct faces if and only if it is 
colored blue, yellow, or light red.  Let $p$ be a point on such a face.  Then
$p$ is on the boundary of two $I$-fibers, on the $V_1$ and $V_2$ sides. Denote
the other endpoints of these fibers by $p_1$ and $p_2$.  By construction we
see that the colors at $p$, $p_1$ and $p_2$ 
fulfill the conditions in Table~\ref{table:colors}.

\begin{table}[htbp]
  \centering
  \begin{tabular}{|l|l|l} \hline %
$p$ & $p_1$, $p_2$  \\ \hline %
blue     &   yellow or light red \\ \hline %
yellow   &   blue or light red \\ \hline %
light red & one is dark red and the other can be any color \\ \hline%
  \end{tabular}\\[3pt]
  \caption{Colors of $I$-adjacent points}
\label{table:colors}
\end{table}

\begin{nota}
\label{nota:LightRed}
Every light red face is $I$-equivalent to a dark red face on one side.  On the
other side it is $I$-equivalent to a face that may be blue, yellow, light red
or dark red.  This decomposes the set of light red points into four disjoint
subsets.  We label a light red face that is $I$-equivalent to a blue
(resp. yellow) face by \lrb\ (resp. \lry).
\end{nota}

\medskip

\noindent {\bf Case Two.  $\alpha_1$ cannot be isotoped to be disjoint from
  $X_2$.}  
Since  $\alpha \subset \mbox{int}(B)$, each point of $\alpha_1$
is yellow or light red.  Hence $\alpha_1$ bounds $I$-bundles on
both sides.  Let $A_{1,2}$ be the be the (possibly immersed) 
$I$-bundle obtained by extending $\alpha_1$
into $V_2$, and denote $\del A_{1,2} \setminus \alpha_1$ by $\alpha_{1,2}$;
see Figure~\ref{fig:a12}.  
\begin{figure}[htbp]
     \centerline{  \includegraphics[width = 50mm]{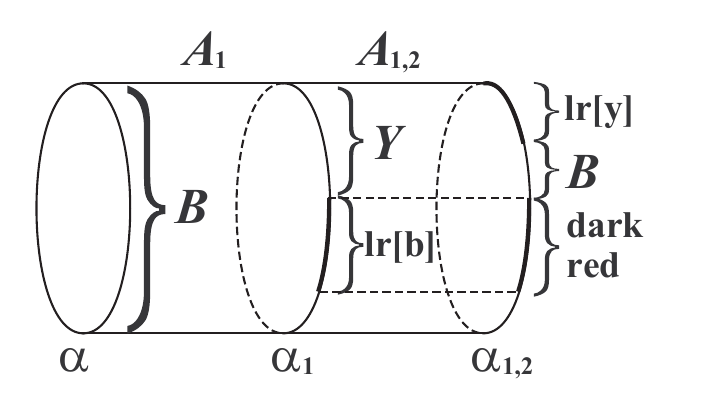}}
  \caption{}
  \label{fig:a12}
\end{figure}
Since every  point of
$\alpha_1$ is yellow or light red and labeled \lrb, every point  of
$\alpha_{1,2}$ is blue, light red and labeled \lry, or dark red  (see
Table~\ref{table:colors} and Notation~\ref{nota:LightRed}).  Thus $\alpha_1
\cap \alpha_{1,2} = \emptyset$, and we see that $A_{1,2}$ is 
trivial $I$-bundle, that is,
an embedded annulus.

Since $X_2$ and $X$ co-bound an $I$-bundle, every point of $X_2$ is yellow or
light red and labeled \lrb.  Thus $\alpha_{1,2} \cap X_2 = \emptyset$.  By
assumption $\alpha_1$ cannot be isotoped off $X_2$.  Hence $\alpha_1$ is not
isotopic to  $\alpha_{1,2}$; this implies that $A_{1,2}$ is not boundary
parallel.  By assumption $A_1$ is not boundary parallel and
$\alpha_1$ is essential in $\s$.  Applying Lemma~\ref{lem:distance}~(1) to
$A_1$, $A_{1,2}$, and $\alpha_1$ we conclude that $d(\s) \leq 2$, completing
the proof of Theorem~\ref{thm}.

\nocite{rr}

\bibliographystyle{amsplain} 

\bibliography{linear}

\end{document}